\documentclass[10pt]{amsart}

\usepackage{amssymb,amsthm,amsmath}

\usepackage{url}

\usepackage{color}
\usepackage{verbatim}
\theoremstyle{plain}

\newtheorem{theorem}{Theorem}[section]
\newtheorem{corollary}[theorem]{Corollary}
\newtheorem{lemma}[theorem]{Lemma}

\theoremstyle{definition}

\newtheorem{definition}[theorem]{Definition}

\newtheorem{remark}[theorem]{Remark}

\newcommand{\R}{{\mathbb R}}

\newcommand{\eps}{\varepsilon}

\newcommand{\PP}{\mathbb{P}} 
\newcommand{\EE}{\mathbb{E}} 

\begin{document}  

\title{On  Khintchine type inequalities for pairwise independent Rademacher random variables}
\footnote{BP is pleased to acknowledge the support of a University of Alberta start-up grant and National Sciences and Engineering Research Council of Canada Discovery Grant number 412779-2012.}
\author{Brendan Pass and Susanna Spektor}
\date{}

\address{Brendan Pass
\noindent Address: University of Alberta, Edmonton, AB, Canada, T6G2G1}
\email{pass@ualberta.ca}

\address{Susanna Spektor
\noindent Address: Michigan State University, East Lansing, MI, USA, 48824 }
\email{sanaspek@gmail.com}


\maketitle

\begin{abstract}
We consider Khintchine type inequalities on the $p$-th moments of vectors of $N$  pairwise independent Rademacher random variables.  We establish that an analogue of Khintchine's inequality cannot hold in this setting with a constant that is independent of $N$; in fact, we prove that the best constant one can hope for is at least $N^{1/2-1/p}$.  Furthermore, we show that this estimate is sharp for exchangeable vectors when $p=4$.  As a fortunate consequence of our work, we obtain similar results for $3$-wise independent vectors.

\medskip

\noindent 2010 Classification: 46B06, 60E15
%

\end{abstract}




\setcounter{page}{1}

\section{Introduction}

Khintchine's inequality is a moment inequality with many important applications in probability and analysis
(see \cite{Garlin, Kah, LT, MS, PShir} among others). It states that  the $L_p$ norm of weighted
independent Rademacher random variables is controlled by their $L_2$ norm; a precise statement follows.
We say that $\eps _0$ is a Rademacher random variable if
$\PP(\varepsilon_0=1)=\PP(\varepsilon_0=-1)=\displaystyle{\tfrac 12}$.
Let $\varepsilon_i$, $i\leq N$,  be independent
copies of $\varepsilon_0$ and $a \in \R^{N}$.
Khintchine's inequality
(see e.g. Theorem 2.b.3 in \cite{LT}  or Theorem 12.3.1 in \cite{Garlin})
states that for any $p\geq 2$ one has
\begin{align}\label{1}
 \left(\EE\left|\sum_{i=1}^{N}a_i\varepsilon_i\right|^p\right)^{\frac 1p}\leq
 C(p) \, \|a\|_{2}=C(p) \left(\EE\left|\sum_{i=1}^{N}a_i\varepsilon_i\right|^2\right)^{\frac 12}.
\end{align}

Note that the constant $C(p)$ does not depend on $N$.  It is natural to ask whether the independence condition can be relaxed; indeed, random vectors with dependent coordinates arise in many problems in probability and analysis (see e.g. \cite{G} and the references therein).  In this short paper, we are interested in what can be said when the independence assumption on the coordinates is relaxed to pairwise (or, more generally, $k$-wise)  independence.  

\begin{definition} We call an $N$-tuple $\varepsilon=\{\varepsilon_i\}_{i\geq 1}^N$ of Rademacher random variables a \textit{Rademacher vector}.
For a fixed non-negative integer $k$, a Rademacher vector is called \emph{$k$-wise independent} if  any  subset  $\{\varepsilon_{i_1}, \varepsilon_{i_2}, \ldots, \varepsilon_{i_k}\}$ of length $k$ is mutually independent.

\end{definition}

When $k=2$ in the preceding definition, we will often use the terminology \emph{pairwise independent} in place of $2$-wise independent.
As it will be useful in what follows, we note that instead of random variables, it is equivalent to consider probability measures $P$ on the set $\{-1,1\}^N$, where $P=law(\varepsilon)$.  The condition that $\varepsilon$ is a Rademacher vector is then equivalent to the condition that the projections $law(\varepsilon_i)$ of $P$ onto each copy of $\{-1,1\}$ is equal to $P_1:=\frac{1}{2}[\delta_{-1} +\delta_1]$.  The $k$-wise independence condition is equivalent to the condition that the projections $law(\varepsilon_{i_1},\ldots,\varepsilon_{i_k})$  of $P$ onto each $k$-fold product $\{-1,1\}^k$ is product measure $\otimes^kP_1$.

In general,  sequences of $k$-wise independent random variables, for $k \geq 2$, share some of the properties of the mutually independent ones, including the second Borel-Cantelli lemma and the strong law of large numbers (see e.x. \cite{law}).   Other properties of mutually independent sequences, such as the central limit theorem, fail to carry over to the pairwise independent setting, however.  
For more on $k$-wise independent sequences and their construction, see, for example \cite{D1, R1, R2}.

Our goal is to determine whether Khintchine's inequality holds for $k$-wise independent Rademacher random variables, and, if not, to understand how badly it fails.  More precisely, if we define


\begin{equation}\label{constant}
C(N,p,k) = \sup_{\substack{a\in \mathbb{R}^N: ||a||_2 =1 \\ \varepsilon \text{ is a k -wise independent Rademacher vector}} }\left(\EE \left|\sum_{i=1}^N a_i\varepsilon_i\right|^p\right)^{1/p},
\end{equation}
then the questions we are interested in can be formulated as:
\begin{enumerate}
\item[1.] Is $C(N,p,k)$ bounded as $N \rightarrow \infty$, for a fixed $p$ and $k$?
\item[2.] If not, what is the growth rate of $C(N,p,k)$?
\end{enumerate}
Note that the $C(N,p,k)$ form a monotone decreasing sequence in $k$, as the $k$-dependence constraint becomes more and more stringent as $k$ increases.  We define $C(N,p,\infty)$ to be the best constant in Khintchine's inequality (for independent random variables):
$$
C(N,p,\infty) = \sup_{a\in \mathbb{R}^N: ||a||_2 =1 }\left(\EE\left|\sum_{i=1}^N a_i\bar \varepsilon_i\right|^p\right)^{1/p}.
$$
where the $\bar \varepsilon_i$ are mutually independent Rademacher random variables.  
Note that, as mutual independence implies $k$-wise independent for any $k$, we have $C(N,p,k) \geq C(N,p,\infty)$.  In this notation, the classical Khintchine inequality means that $C(N,p,\infty)$ is bounded as $N$ goes to $\infty$ for each fixed $p$.

Some properties of $C(N,p,k)$ are easily discerned.  By an application of Holder's inequality, we get, for any Rademacher $\varepsilon$ and any $a$,

\begin{equation}\label{holderbound}
 \left(\EE\left|\sum_{i=1}^N a_i\varepsilon_i\right|^p\right)^\frac{1}{p} \leq \sqrt{N} ||a||_2
\end{equation}
and so we  have
$$
C(N,p,k) \leq \sqrt{N}.
$$
In fact, for a random vector with $law(\varepsilon)=\frac{1}{2}[\delta_{1,1,1,\ldots,1} +\delta_{-1,-1,-1,\ldots,-1} ]$, we get equality in \eqref{holderbound}.  This $\epsilon$ is $1$-wise independent, which simply means that it has Rademacher marginals and so we have $C(N,p,1) =\sqrt{N}$.  Clearly, this vector is not pairwise independent, however, and so it provides no further information on $C(N,p,k)$ for $k \geq 2$.

Let us also mention that, when $p$ is an even integer, and $k  \geq p$, it is actually a straightforward calculation to show that $C(N,p,k)=C(N,p,\infty) \approx 1$ is independent of $k$ (that is, Khintchine's inequality for $k$-wise independent random variables holds with the same constant as in the independence case).  This seems to be a "folklore" result, which is well known to experts, but we were unable to find a suitable reference.

In this paper, we focus on the $k=2$ case.  We prove that for $p \geq 2$ and $N$ even, $C(N,p,2) \geq N^{1/2-1/p}$, providing a negative answer to the first question above. Moreover, if we define $C_e(N,p,k)$ as in \eqref{constant}, but with the supremum restricted to \emph{exchangeable} Rademacher vectors $\varepsilon$,  and consider the $p=4$ case, we prove that $C_e(N,4,2) = N^{1/2-1/4} =N^{1/4}$.

As a fortunate consequence of our work here, we obtain analagous results for $k=3$. Understanding the $k \geq 4$ case remains an interesting open question.


\section{Estimates on $C(N,p,2)$}
The following Theorem provides a negative  answer to the first question in the introduction when $k=2$, and provides some information about the second question in the same case.
\begin{theorem}\label{main}
Let $N=2n$ be even and set $a=(1,1,....1) \in \mathbb{R}^N$.  Then, for all $p \geq 2$,

\begin{equation}\label{theorem}
\sup_{ \varepsilon \text{ is a k -wise independent Rademacher vector} }\left(\EE \left|\sum_{i=1}^N a_i\varepsilon_i\right|^p\right)^{1/p}= N^{1/2-1/p} \|a\|_2
\end{equation}

Consequently,
 $C(N, p, k) \geq N^{1/2-1/p}$. 


\end{theorem}
\begin{proof}
We explictly construct a pairwise independent Rademacher vector $\varepsilon=(\varepsilon_1,...\varepsilon_N)$  for which we get equality in \eqref{theorem}, and then show that this maximizes the left hand side over the set of all such vectors.
To do this, we will define the probability measure $P=law(\varepsilon)$.

We define $P=\frac{1}{N} P_a+ \frac{N-1}{N} P_b$ where $P_a=\frac{1}{2}[\delta_{1,1,1,\ldots,1} +\delta_{-1,-1,-1,\ldots,-1}]$ is uniform measure on the two points $(1,1,1,\ldots,1), (-1,-1,\ldots,-1) \in \{-1,1\}^N$ and $P_b$ is uniform measure on the set of all points with an equal number of $1$'s and $-1$'s; that is, points which are permutations of $\{{\underbrace{1, 1, \ldots, 1}_{N/2 \, \textit{of them}}\underbrace{-1, -1, \ldots, -1}_{N/2 \, \textit{of them}}}\}$.
We first verify that this is pairwise independent probability measure; that is, that it's twofold marginals are $\frac{1}{4}(\delta_{1,1} +\delta_{1,-1} +\delta_{-1,1} +\delta_{-1,-1})$.  By symmetry between the coordinates, it suffice to verify this fact for the projection $P_2$ on the first two copies of $\{-1,1\}$.  To see this, we have

$$
P_2(1,1) = \frac{1}{N}P_a(1,1,1,\ldots,1) +\frac{N-1}{N}P_b\{\varepsilon:(\varepsilon_1,\varepsilon_2)=(1,1)\}.
$$
Now, $P_a(1,1,1,\ldots,1) =\frac{1}{2}$, and it  is easy to see that $P_b\{\varepsilon:(\varepsilon_1,\varepsilon_2)=(1,1)\} =\frac{N/2-1}{2(N-1)}$, implying
$$
P_2(1,1) = \frac{1}{2N} +\frac{(N-1)(N/2-1)}{2N(N-1)}=\frac{1}{4}.
$$
Similar calculations imply $P_2(1,-1) =  P_2(-1,1)=P_2(-1,-1)=\frac{1}{4}$, and so $P$ is pairwise independent.

Now, letting $\varepsilon =(\varepsilon_1,...\varepsilon_N)$ be a random variable with $law(\varepsilon) =P$, and noting that $ |\sum_{i=1}^N\varepsilon_i|^p$ is $0$ for points in the support of $P_b$ and $N$ for points in the support of $P_a$, we have

$$
\EE \left|\sum_{i=1}^N\varepsilon_i\right|^p =\frac{1}{N}N^p=N^{p-1}
$$
Noting that $||a||_2 =\sqrt{N}$, it follows that
$$
\left[\EE\left|\sum_{i=1}^N\varepsilon_i\right|^p\right] ^{1/p}=N^{1-1/p}=\sqrt{N} N^{1/2-1/p}=||a||_2N^{1/2-1/p}
$$

 It remains to  show that the $\varepsilon$ we have constructed is optimal in \eqref{theorem} .   Define the functions $u:\{1,-1\}^2 \rightarrow \mathbb{R}$ by
$$
u(1,1)=u(-1,-1)=N^p/{N \choose 2}
$$

and

 $$
u(1,-1) =u(-1,1) = -\frac{N-2}{N}u(1,1)= -\frac{(N-2)N^{p-1}}{{N \choose 2}}.
$$
We will  show that for any $(\varepsilon_1,\ldots,\varepsilon_N) \in \{  1,-1 \}^N$, we have

\begin{equation}\label{potentialinequality}
\left|\sum_{i=1}^N\varepsilon_i\right|^p \leq \sum_{i < j}^Nu(\varepsilon_i,\varepsilon_j)
\end{equation}
with equality only on the points in the support of $P$; that is, at $(1,1,....,1), (-1,-1,....,-1)$ and permutations of  $\{{\underbrace{1, 1, \ldots, 1}_{N/2 \, \textit{of them}}\underbrace{-1, -1, \ldots, -1}_{N/2 \, \textit{of them}}}\}$.

For $(\varepsilon_1,\ldots,\varepsilon_N) \in \{  1,-1 \}^N$, let $n$ be the number of $\varepsilon_i$'s which are equal to positive $1$.  The  inequality \eqref{potentialinequality} which we are attempting to establish is equivalent to

$$
|n-(N-n)|^p \leq {n \choose 2} u(1,1)+ {N-n \choose 2}u(-1,-1) +n(N-n)u(1,-1)
$$
or

\begin{eqnarray*}
|2n-N|^p& \leq&\left[ {n \choose 2} + {N-n \choose 2} -n(N-n)\frac{N-2}{N}\right]u(1,1)\\
&=&\left[ \frac{n(n-1)}{2} + \frac{(N-n) (N-n-1)}{2} -n(N-n)\frac{N-2}{N}\right]\frac{2N^p}{N(N-1)}\\
\left|\frac{2n-N}{N}\right|^p& \leq &\left[ \frac{n(n-1)}{2} + \frac{(N-n) (N-n-1)}{2} -n(N-n)\frac{N-2}{N}\right]\frac{2}{N(N-1)}
\end{eqnarray*}
Now, a straightforward calculation shows that the right hand side is equal to $\displaystyle{\left|\frac{2n-N}{N}\right|^2}$.  As clearly $0 \leq n \leq N$, we have that $\displaystyle{\left|\frac{2n-N}{N}\right| \leq 1}$.  As $ p \geq 2$, we then clearly have $\displaystyle{\left|\frac{2n-N}{N}\right|^p \leq \left|\frac{2n-N}{N}\right|^2}$, with equality only when $2n-N =0,N$ or $-N$. That is, we have equality precisely when $n=0,\frac{N}{2}$ or $N$, which correspond exactly to points in the support of $P$.  This establishes \eqref{potentialinequality}, with equality only when $\epsilon$ is in the support of $P$.

Now note that by \eqref{potentialinequality} for any pairwise independent  Rademacher vector $\tilde \varepsilon$ on $\{  1,-1 \}^N$, we have

$$
\EE\left|\sum_{i=1}^N\tilde \varepsilon_i\right|^p \leq\EE \big(\sum_{i \leq j}^Nu(\tilde \varepsilon_i,\tilde \varepsilon_j) \big)= \sum_{i \leq j}^N\EE \big(u(\tilde \varepsilon_i,\tilde \varepsilon_j)\big)
$$
But the right hand side is constant on the set of pairwise independent Rademacher vectors, as it depends only on the twofold vectors $(\tilde \varepsilon_i,\tilde \varepsilon_j)$.  As we have equality in \eqref{potentialinequality} $P$ almost surely, for the particular sequence $\varepsilon$ with $law(\varepsilon)=P$, we get

$$
\EE\left|\sum_{i=1}^N \varepsilon_i\right|^p  =\EE \big(u( \varepsilon_i, \varepsilon_j)\big),
$$
and therefore,
$$
\EE\left|\sum_{i=1}^N\tilde \varepsilon_i\right|^p \leq \EE\left|\sum_{i=1}^N\varepsilon_i\right|^p
$$
completing the proof.

\end{proof}
\begin{remark}

It is worth mentioning that the linear program \eqref{theorem} is reminiscent of a discrete version of the multi-marginal optimal transport problem (see \cite{P} and the references therein).  The difference is that here, the twofold marginals $P_2 =law (\varepsilon_i,\varepsilon_j)$ are prescribed, rather than the marginal $P_1 =law(\varepsilon_i)$.  Moreover, the function $u$ that shows up in the proof can be interpreted in terms of the dual  program, which is is to minimize

$$
\EE\left(\sum_{i \neq j}u_{ij} (\varepsilon_i,\varepsilon_j)\right)
$$
over all collections of functions $u_{ij}:\{-1,1\}^2 \rightarrow \mathbb{R}$, for $1\leq i\neq j\leq N$, which satisfy the constraint
$$
\sum_{i \neq j}u_{ij} (\varepsilon_i,\varepsilon_j) \geq \left|\sum_{i=1}^N\varepsilon_i\right|^p
$$
for all $\varepsilon =(\varepsilon_1,\ldots,\varepsilon_N) \in \{-1,1\}^N$.  The minimizing functions for this dual program are exactly $u_{ij}=u$ for all $i,j$, where  $u$ is as in the proof above.


\end{remark}

Next, we turn our attention to the second question in the introduction, and try to determine the precise behaviour of $C(N,p.k)$.  For this, we study only the $k=2, p=4$ case, and restrict our attention to exchangeable Rademacher vectors.

\begin{lemma}
For any $a \in \mathbb{R}^N$ with $||a||_2=1$ and any pairwise independent, exchangeable $\varepsilon$, we have

$$
\EE\left(\left|\sum_{i=1}^Na_i\varepsilon_i\right|^4 \right)\leq \max\left\{\EE\left(\left|\sum_{i=1}^N\frac{1}{\sqrt{N}}\varepsilon_i\right|^4\right), \EE\left(\left|\sum_{i=1}^Na_i\bar \varepsilon_i\right|^4\right)\right\}
$$
where $\bar \varepsilon$ denotes an independent Rademacher vector.
\end{lemma}
\begin{proof}
We assume, without loss of generality, that $a_i \geq 0$ for each $i$.  After expanding the power, and noting that the Rademacher
condition implies $\varepsilon_i^2=1$ almost surely, the pairwise independence condition implies
 \begin{eqnarray*}
\EE(\varepsilon_i\varepsilon_j\varepsilon_k^2)=\EE(\varepsilon_i\varepsilon_j) =\EE(\bar \varepsilon_i \bar \varepsilon_j)=\EE(\bar \varepsilon_i\bar \varepsilon_j \bar\varepsilon_k^2)\\
\EE(\varepsilon_i\varepsilon_j^3)=\EE(\bar \varepsilon_i \bar \varepsilon_j^3)\\
\EE(\varepsilon_i^2\varepsilon_j^2)=\EE(\bar \varepsilon_i^2 \bar \varepsilon_j^2)\\
\EE(\varepsilon_i^4)=\EE(\bar \varepsilon_i^4 ).
\end{eqnarray*}
Therefore, we have
\begin{equation}
\EE\left(\left|\sum_{i=1}^Na_i\varepsilon_i\right|^4 \right) =\EE\left(\left|\sum_{i=1}^Na_i\bar \varepsilon_i\right|^4 \right)+\sum_{i,j,k,l\text{  distinct}}a_ia_ja_ka_l \EE\left(\varepsilon_i\varepsilon_j\varepsilon_k\varepsilon_l\right)
\end{equation}

By exchangability,  $\displaystyle{\EE\left(\varepsilon_i\varepsilon_j\varepsilon_k\varepsilon_l\right)}$ is independent of $i,j,k$ and $l$; denoting  $\EE\left(\varepsilon_i\varepsilon_j\varepsilon_k\varepsilon_l\right):=c$, this gives us
\begin{equation}\label{expand}
\EE\left(\left|\sum_{i=1}^Na_i\varepsilon_i\right|^p \right) =\EE\left(\left|\sum_{i=1}^Na_i\bar \varepsilon_i\right|^p \right)+c\sum_{i,j,k,l\text{  distinct}}a_ia_ja_ka_l .
\end{equation}
Now, if $c \leq 0$, the above is less than  $\EE \left(\left|\sum_{i=1}^Na_i\bar \varepsilon_i\right|^p\right)$, and the proof is complete.  If, on the other hand, $c \geq 0$, we have, by Maclaurin's inequality,
\begin{eqnarray}
\sum_{i,j,k,l\text{  distinct}}a_ia_ja_ka_l &\leq&4!{ N \choose 4}\left[\frac{1}{2!{N \choose 2}} \sum_{i \neq j}a_ia_j\right]^2\notag\\
&\leq&4!{ N \choose 4}\left[\frac{1}{4{N \choose 2}} \sum_{i \neq j}(a_i^2 +a_j^2)\right]^2\notag\\
&=&24{ N \choose 4}\left[\frac{1}{4{N \choose 2}} 2(N-1)||a||_2^2\right]^2\notag\\
&=&24{ N \choose 4}\frac{1}{N^2}.\label{quarticbound}
\end{eqnarray}
We have equality when all of the $a_i =\frac{1}{\sqrt{N}}$.  Note that this implies that

\begin{equation}\label{equalelements}
\EE\left(\left|\sum_{i=1}^N\frac{1}{\sqrt{N}}\varepsilon_i\right|^4 \right)=\EE\left(\left|\sum_{i=1}^N\frac{1}{\sqrt{N}}\bar \varepsilon_i\right|^4 \right) +c24{ N \choose 4}\frac{1}{N^2}
\end{equation}

Now, it is well known that  $\displaystyle{\EE\left(\left|\sum_{i=1}^Na_i\bar \varepsilon_i\right|^4\right) \leq \EE\left(\left|\sum_{i=1}^N\frac{1}{\sqrt{N}}\bar \epsilon_i\right|^4\right)}$, and so plugging this and \eqref{quarticbound} into \eqref{expand}, we have

\begin{eqnarray}
\EE\left(\left|\sum_{i=1}^Na_i\varepsilon_i\right|^4 \right) &\leq& \EE\left(\left|\sum_{i=1}^N\frac{1}{\sqrt{N}}\bar \varepsilon_i\right|^4 \right) +c24{ N \choose 4}\frac{1}{N^2}\\
&=&\EE\left(\left|\sum_{i=1}^N\frac{1}{\sqrt{N}}\varepsilon_i\right|^4 \right),
\end{eqnarray}
where the last line follows from \eqref{equalelements}.  This completes the proof.
\end{proof}
Theorem \ref{main}, the classical Khintchine inequality, and the preceding lemma now imply the following variant of Khintchine's inequality, for $p=4$ and pairwise independent, exchangeable Rademacher random vectors.
\begin{corollary}
We have $C_{e}(N,4,2) = N^{1/4}$.  In particular, for any  $a \in \mathbb{R}^N$  and any pairwise independent exchangeable Rademacher vector $\epsilon$, we have

$$
\left(\EE\left(\left|\sum_{i=1}^Na_i\varepsilon_i\right|^4 \right) \right)^{\frac{1}{4}} \leq N^{\frac{1}{4}}||a||_2.
$$
\end{corollary}

\begin{remark}
It is straightforward to verify that the measure $P=law(\varepsilon)$ derived in  Theorem is in fact $3$-wise independent, and thus we immediately obtain analogues of the preceding results for $k=3$:

$$
C(N,p,3) \geq N^{1/2-1/p}
$$
and
$$
C_{e}(N,4,3) = N^{1/4}.
$$
\end{remark}
\begin{remark}
Generally speaking, one can identify exchangeable,   $k$-wise independent random variables $x_1, \ldots, x_N$ on $\mathbb{R}$ having equal fixed marginals $P_1=law(x_i)$ with permutation symmetric probability measures $P=law(x_1, \ldots, x_N)$ on $\mathbb{R}^N$ whose $k$-fold marginals are $\otimes ^k P_1$.  The set of measures satifying these constraints is a convex set, and identifying the set of extremal points, or vertices, of this set is an interesting and nontrivial question.

It is easy to see upon inspection of the proof of Theorem 2.1 that the measure  $law(\varepsilon)$ we construct is the unique maximizer of the linear functional $law(\varepsilon)\mapsto \EE(\sum_{i=1}^N\varepsilon_i)$ on the convex set of symmetric, pairwise independent Rademacher probability measures.  As a consequence of this proof, we have therefore identified an extremal point of this set.

\end{remark}





\end{document}